\documentclass{amsart}

\usepackage{fancyvrb}
\usepackage{hyperref}
\usepackage{enumerate}
\usepackage{amssymb}

\newtheorem{theorem}{Theorem}[section]
\newtheorem{lemma}[theorem]{Lemma}
\newtheorem{proposition}[theorem]{Proposition}

\theoremstyle{definition}

\theoremstyle{remark}
\newtheorem{remark}[theorem]{Remark}

\numberwithin{equation}{section}

\def\cD{ {\mathcal D} }

\def\cH{ {\mathcal H} }

\def\cK{ {\mathcal K} }

\def\cP{{\mathcal P}}

\def\cT{{\mathcal T}}

\def\cZ{ {\mathcal Z} }

\def\R{{\mathbb R}}

\def\hcK{\hat{\cK}}

\def\smatn{\mathbb S_n(\mathbb R)}
\def\smatng{\mathbb S_n(\mathbb R^g)}
\def\smatmg{\mathbb S_m(\mathbb R^g)}

\def\smatn{\mathbb S_n(\mathbb R)}

\def\Rx{\mathbb R\langle x \rangle}

\newcommand{\df}[1]{{\bf{#1}}{\index{#1}}}

\makeindex

\begin{document}
\title[Free and quasi-convex]{Quasi-Convex Free Polynomials}
\author{S. Balasubramanian and S. McCullough${}^1$}
\thanks{${}^1$ Research supported by NSF grants DMS 0758306 and 1101137}
\address{Department of Mathematics and Statistics,
Indian Institute of Science Education and Research (IISER) - Kolkata, Mohanpur Campus,
Nadia District, Pin: 741246, West Bengal, India.}
\email{bsriram@iiserkol.ac.in}
\address{Department of Mathematics\\
The University of Florida\\
Box 118105\\
Gainesville, FL 32611-8105 }
\email{sam@ufl.edu}


\begin{abstract}
 Let $\Rx$ denote the ring
 of polynomials in $g$ freely non-commuting variables
  $x=(x_1,\dots,x_g)$.
 There is a natural involution $*$ on $\Rx$
 determined by $x_j^*=x_j$ and $(pq)^*=q^* p^*$
 and a free polynomial $p\in\Rx$ is symmetric if it is invariant
 under this involution. If $X=(X_1,\dots,X_g)$ is a $g$ tuple of
  symmetric $n\times n$ matrices, then the evaluation $p(X)$
  is naturally defined and further $p^*(X)=p(X)^*$.
  In particular, if $p$ is symmetric, then $p(X)^*=p(X)$.
  The main result
  of this article says if $p$ is symmetric, $p(0)=0$
 and for each $n$ and each symmetric positive
  definite $n\times n$
  matrix $A$ the set $\{X:A-p(X)\succ 0\}$ is convex, then
  $p$ has degree at most two and is itself convex,
  or $-p$ is a hermitian sum of squares.
\end{abstract}

\maketitle

\section{Introduction}
  Let $\Rx$ denote the ring of polynomials over $\mathbb R$
  in the freely non-commuting variables $x=(x_1,\dots,x_g)$.
  A $p\in\Rx$ is a \df{free polynomial} and is a finite sum
\[
  p=\sum p_w w,
\]
  over words $w$ in $x$ with coefficients $p_w\in\mathbb R.$
  The empty word, which plays the role of
  the multiplicative identity,  will be denoted $\emptyset$.

  For a word
\begin{equation}
 \label{eq:word}
  w=x_{j_1}x_{j_2}\cdots x_{j_k},
\end{equation}
 let
\[
  w^* = x_{j_k} \cdots x_{j_2}x_{j_1}.
\]
 The operation ${}^*$ extends naturally to an involution on
  $\Rx$ by
\[
  p^* = \sum p_w w^*.
\]

 Let $\smatng$ denote the set of $g$-tuples
  $X=(X_1,\dots,X_g)$ of $n\times n$ symmetric matrices.
 For a word $w$ as in \eqref{eq:word}, substituting $X_j$ for $x_j$ gives,
\[
 X^w = w(X)=X_{j_1} X_{j_2}\cdots X_{j_k}.
\]
  This evaluation extends to $\Rx$ in the obvious way,
\[
 p(X)=\sum p_w w(X).
\]
 Observe,  for $0\in\smatng$ that
  $p(0)=p_{\emptyset} I_n$, where $I_n$ is the $n\times n$ identity.

 The transpose operation ${}^*$ on matrices
 is compatible with the involution ${}^*$ on $\Rx$ in that
\[
  p(X)^*=p^*(X).
\]
  A polynomial $p\in\Rx$ is symmetric if $p=p^*$ and in this
  case $p(X)^*=p^*(X)=p(X)$ so that $p$ takes symmetric values.

 Let $\smatn$
 denote the collection of symmetric $n\times n$ matrices.
 Given $S\in\smatn$ the notations $S\succ 0$ and
  $S\succeq 0$ indicate that $S$ is positive definite and
  positive semidefinite respectively.
 A symmetric $p\in\Rx$ is
 {\it  quasi-convex}  if $p(0) = 0$ and
 for each $n$ and positive definite matrix
 $n\times n$ matrix $A$  the set
\[
 \cD(A) =\{X\in\smatng: A-p(X) \succ 0 \}
\]
 is convex.

 A symmetric polynomial $p$ is a  {\it (hermitian) sum of squares}
 if there exists an $m$ and $h_1,\dots,h_m\in \Rx$ such that
\[
  p=\sum h_j^* h_j.
\]
 Evidently such a $p$ is positive in the sense that for each $n$ and
 $X\in\smatng$,
\[
  p(X)\succeq 0.
\]
  The following theorem is the main result of this article.

\begin{theorem}
 \label{thm:main}
  If $p$ is quasi-convex, then either $-p$ is
  a sum of squares, or
  there exists a linear polynomial $\ell\in\Rx$
   and finitely many linear polynomials $s_j\in\Rx$ such that
\begin{equation}
 \label{eq:wow}
  p(x)= \ell(x) +\sum s_j^*(x) s_j(x).
\end{equation}
  Thus, $p$ is a hermitian sum of squares of linear polynomials  plus a linear term.

  Further, if there is an $N$ such 
  that for each $n\ge N$ there is a 
 % If it is also assumed that for
 % sufficiently large $n$ there exist
  $B\in\smatn$ such that $B\not\succeq 0$
  and $\{X\in\smatng: B-p(X)\succ 0\}$
  is convex, then $-p$ is a sum of squares
  if and only if $p=0$.
\end{theorem}

\begin{remark}\rm
 It is easy to see that if $p$ is a hermitian sum of squares
 of linear polynomials   plus a
 linear term, then $p$ is quasi-convex.
\end{remark}

\subsection{Related results and remarks}
  Theorem \ref{thm:main}
  falls within the emerging fields of
  free analysis and free semialgebraic geometry.
  Free semialgebraic geometry is, by analogy to the
  commutative case, the study of free polynomial inequalities.
  For instance there are now a number of free Positivstellensatze
  for which \cite{C} \cite{KS1} \cite{KS2}  \cite{HMP}
  are just a few references.
  In this regard, see also \cite{Sima} and its Proposition 17.
 % The tutorial \cite{HKM} gives a
 % gentle introduction to the subject.
  There is also a theory
  of free rational functions.  Recent developments in this
  direction have been related to non-commutative
  multi-variate systems theory. See for instance
  \cite{BGM}.  Free rational functions actually
  appeared much earlier in the context of finite automata.
  See for instance \cite{Shu}.
   Issues of convexity in the
  context of free polynomials and rational functions
   naturally arise in systems theory problems
  \cite{HMPV,OHMP} and mathematically are related to the theory
  of operator spaces and systems and matrix convexity
  \cite{EW}.
  More generally, there is a theory of free analytic
  functions which arise naturally in several contexts,
  including free probability.  A sampling of references includes
   \cite{Voic} \cite{DKV-VV} \cite{Pop} and
  \cite{MS}.

  Some systems theory problems present as a free polynomial inequality
  (or more realistically as a system of matrix-valued
   free polynomial inequalities)
  involving two classes of free (freely non-commuting)
  variables, say the $a$ variables
  and the $x$ variables.  The $a$ variables are thought of as
  known (system or state) parameters and the $x$ variables unknowns.
  For a given free polynomial $q$,
  of interest is the case that, for each fixed $A$ in
  some distinguished collection of
  known parameters,  the inequality $q(A,x)\succ 0$
  is convex in $x$.  Thus this article  considers the simplest such
  case. Namely there is just one $a$ variable and
  $q(a,x)=a-p(x)$ for a polynomial $p$ in the variables $x$ alone.
  For comparison, a main result of \cite{HHLM}
  says, generally, if $q(A,x)$ is convex in $x$ for each fixed $A$, then
  $q(a,x) = L(a,x) + \sum h_j(a,x)^* h_j(a,x)$ where
   $L$ has degree at most one in $x$ and each $h_j$
  is linear in $x$.  The articles
  \cite{DHM1,DHM2,DHM3,HM} contain results for polynomials $f$
  whose positivity set - namely the set of those $X$ such that $f(X)\succ 0$ -
  is convex.

  A symmetric polynomial $p$ is \df{matrix convex} if for
  each $n,$  each pair $X,Y\in\smatng,$ and each $0\le t\le 1$,
\[
 p(tX+(1-t)Y)\preceq tp(X)+(1-t)p(Y).
\]
  The following Theorem, pointed out by the referee, generalizes
  the main result of \cite{HM2}.

\begin{theorem}
 \label{thm:jc}
   For a symmetric polynomial $p$ the following are equivalent.
 \begin{enumerate}[(i)]
  \item $p-p(0)$ has the form in Equation \eqref{eq:wow};
  \item $p$ is matrix convex;
  \item $\cD(A)$ is convex for every $A\in\smatng$;
  \item $p(x)-p(0)$ is quasi-convex and $p(0)-p(x)$ is not
   a  nonzero sum of squares.
 \end{enumerate}
\end{theorem}

 A remark is in order about the definition of quasi-convex used here.
 A classical definition says that a function $f$ of several
 (commuting) variables is quasi-convex if each of its sub-level
 sets is convex. (The interested reader can work out
  the relationships between this definition of quasi-convex
  and the seemingly more popular one
  $f(tx+(1-t)y)\le \max\{f(x),f(y)\}$ for $0\le t\le 1$.)
  In considering a free analog here,  because of the role that positivity
  plays in the arguments, it was convenient to make the
  harmless normalization that $p(0)=0$ and then only require
  convexity of the (level) sets $\cD(A)$ for $A$ positive definite
  in the definition of quasi-convex.

\subsection{Reader's guide}
The remainder of this paper is organized as follows.
 Issues surrounding sums of squares are dealt
with in Section \ref{sec:SoS}.
Just as in the commutative case, convexity is related
to positivity of a Hessian.  The necessary definitions and basic results
appear in Section \ref{sec:DDH}.
Section \ref{sec:boundary} examines membership in the boundary of
the set $\cD(A)$ as well as consequences of the
convexity hypothesis. Theorem \ref{thm:main}
is proved in the final section, Section \ref{sec:DSLI}.

We thank the referee for many corrections, suggestions and
references related to this article.

%--------------------------

\section{The Sum of Squares Case}
\label{sec:SoS}
  The following proposition dispenses with the
  alternative that $-p$ is a sum of squares.

\begin{proposition}
  If there is an $N$
  such that for each $n\ge N$ there exists a $B\in \smatn$
  such that $B\not\succeq 0$ and the set $\cD(B)$ is convex,
  then $-p$ is not a nonzero sum of squares.
\end{proposition}

\begin{proof}
  Arguing by contradiction, suppose $-p$ is a sum of squares.
  Consider the polynomial
\[
 -p(tx) = \sum_{j=1}^{2d} p_j(x) t^j,
\]
 where the $p_j$ are homogeneous of degree $j$ polynomials in
 the free variables $x$.  Since $-p$ is a sum of squares,
 $-p(tx)$ has even degree as a polynomial in $t$,
 and so, without loss of generality,
 we may assume that $p_{2d}$ is nonzero. Now $p_{2d}$
 is itself a sum of squares and hence it takes positive
 semi-definite values.
 If, on the other hand, $p_{2d}(X)$ is never positive definite,
 then $\det(p_{2d}(X))=0$  for all $n$ and $X\in\smatng$.
 An application of the Guralnick-Small
 lemma as found in \cite{HHLM} then gives the contradiction
 that $p_{2d}$ is the zero polynomial.
  Thus there is an $n$
 and an $X\in\smatng$ such that $p_{2d}(X)\succ 0$.
  By assumption, there is a $B\in\smatn$ such that $B\not\succeq 0$
  and $\cD(B)$
  is convex.  Choosing
 $t$ sufficiently large, it may be assumed that both
  $B-p(tX)$ and $B-p(-tX)$ are positive definite. In this
  case $tX$ and $-tX$ are in $\cD(B)$ and thus
  $0= \frac12 (tX+(-tX)) \in \cD(B)$, contradicting
  the assumption that $B$ is not positive semidefinite
  and completing the proof.
\end{proof}

\subsection{When $-p$ is not a sum of squares}
 Given $n$, let $\cK(n)$ denote the set of those $X\in\smatng$
 such that  $p(X)$ has a positive eigenvalue.
 The $\cK(n)$ are open sets, the issue is whether they are empty or not.
 In the free setting, and unlike in the case of several commuting variables,
  positive polynomials are sums of squares
  with \cite{S} \cite{Sima} and \cite{H} as a very small
  sampling of the references.    (For a reference which
  explicitly treats the case of the symmetric variables ($x_j^*=x_j$)
  used here see \cite{M}.)
  In particular, each $\cK(n)$ is empty
  if and only if $-p$ is a sum of squares.

%  a rephrasing of a main result from \cite{M}
% says that each $\cK(n)$ is empty if and only if $-p$ is
%  a sum of squares.)
 The conclusion of the following
 lemma will be used later in the proof of Theorem \ref{thm:main}
  when $-p$ is not a sum of squares.

\begin{lemma}
 \label{lem:reduce}
  Suppose $q\in\Rx$ and $-p$ is not a sum of squares.
  If
  $\det(q(X))=0$ for every $n$ and  $X\in \cK(n)$, then
  $q=0$.
\end{lemma}

  Before beginning the proof of Lemma \ref{lem:reduce}
  we record the following simple fact.

\begin{lemma}
 \label{lem:easy}
 Suppose $q\in\Rx$ and $m$ is a positive integer.
  If $\cK(m)$ is nonempty and
  $\det(q(X))=0$ for all $X\in \cK(m)$, then
  $\det(q(X))=0$ for all $X\in\smatmg$.
\end{lemma}

\begin{proof}
  The function $\smatmg \ni X\mapsto \det(q(X))$ is
  a polynomial in the entries of $X$. Hence, if it
  vanishes on an open set it must be identically zero.
\end{proof}

  Given integers $k,\ell$. Let $m=k+\ell$ and consider the subspace
\[
  S=\mathbb S_{k}(\mathbb R^g) \oplus \mathbb S_\ell(\mathbb R^g)
\]
 of $\smatmg$. Each tuple $X\in S$ is a direct sum
 $X=(Y_1\oplus Z_1,\dots, Y_g\oplus Z_g)$ where
 $Y=(Y_1,\dots,Y_g)\in \mathbb S_{k}(\mathbb R^g)$ and
  $Z=(Z_1,\dots,Z_g)\in \mathbb S_{\ell}(\mathbb R^g)$ and where
\[
  Y_j\oplus Z_j =\begin{pmatrix} Y_j &0 \\ 0 & Z_j\end{pmatrix}.
\]

\begin{proof}[Proof of Lemma \ref{lem:reduce}]
 %Since $-p$ is not a sum of squares,
 % there exists a positive integer $m$
 % such that $\cK(m)$ is not empty.
  Since $-p$ is not a sum of squares,
  by the remarks at the outset of this section,
  there is an $m$ and a $Y\in \cK(m)$ such that $p(Y)$ has a positive eigenvalue.
 First observe that for any positive integer $k$ and
 $X\in\mathbb S_k(\mathbb R^g)$ that $p(X\oplus Y)=p(X)\oplus p(Y)$
 too has a positive eigenvalue. Thus, $\cK(n)$ is nonempty
  for each $n\ge m$. Lemma \ref{lem:easy} now implies
 that $\det(q(X))=0$ for each $n\ge m$ and $X\in\smatng$.

  Now suppose $1\le \ell <m$. Since $n=m\ell \ge m$,
   $\det(q(X))=0$ for  all $X\in\smatng.$
 Given $X\in\mathbb S_\ell(\mathbb R^g)$, the tuple
\[
 \tilde{X}=\oplus_1^m X \in \smatng.
\]
 Thus $(\det(q(X)))^m=\det(q(\tilde{X}))=0$.  Hence $\det(q(X))=0$
  for every $X\in \mathbb S_\ell (\mathbb R^g)$.

 Since $\det(q(X))=0$ for every $n$ and $X\in\smatng$,
 the   Guralnick-Small lemma in \cite{HHLM} implies  $q$ is the zero
  polynomial.
\end{proof}

%-------------------------

\section{Directional Derivatives and the Hessian}
 \label{sec:DDH}
 Given $p\in\Rx$ another set of freely non-commuting
 variables  $h=(h_1,\dots,h_g)$ and the real parameter $t$,
\begin{equation*}
 %\label{eq:def-dervs}
   p(x+th)= \sum p_j(x)[h] t^j,
\end{equation*}
 where $p_j(x)[h]$ are polynomials in the
 variables $(x,h)=(x_1,\dots,x_g,h_1,\dots,h_g)$
  (which are of course freely non-commuting). The notation
 indicates the different role that these variables play. Indeed,
  observe that $p_j(x)[h]$ is homogeneous of degree $j$ in $h$.

  The polynomial $p_1(x)[h]$ is the {\it directional derivative}
  or simply the {\it derivative}  of  $p$
  (in the direction $h$) and is denoted $p^\prime(x)[h]$.
  The polynomial $2 p_2(x)[h]$ is the {\it Hessian} of $p$
  and is denoted by $p^{\prime \prime}(x)[h]$.

  Given $X\in\smatng$ and $v\in\mathbb R^n$, let
\begin{equation*}
 % \label{eq:clamped-tangent}
  \cT(X,v)= \{H\in\smatng : p^\prime(X)[H]v=0 \} \subset \smatng.
\end{equation*}
  In the case that $(A-p(X))v=0$, the subspace $\cT(X,v)$ is the
 {\it clamped tangent plane} to $\cD(A)$ at $(X,v)$ \cite{DHM1}.
  In this case, if one chooses $H\in\cT(X,v)$, then
\[
 \langle (A-p(X+tH))v,v\rangle = - \frac{1}{2}
       t^2 \langle p^{\prime\prime}(X)[H]v,v\rangle
    +t^3 e(t),
\]
 for some polynomial $e(t)$. This identity, much as in the commutative
  case, provides a link between convexity and positivity
  of the Hessian of $p$.

%---------------------------------------

\section{The boundaries}
\label{sec:boundary}
  Fix $p$ satisfying the hypothesis of Theorem \ref{thm:main}.
  In particular $p(0)=0$.

\begin{lemma}
   \label{lem:closint}
     Let $n$ and a positive
     definite $A\in\smatn$ be given. A given
     $X\in\smatng$ is in the boundary of $\cD(A)$
     if and only if $A-p(X)$ is positive semidefinite
     and has a non-trivial kernel.
 \end{lemma}

\begin{proof}
  Suppose that $X$ is in the boundary of $\cD(A)$. It follows that
  $A - p(X) \succeq 0$. It must be the case that $A - p(X)$ has
  a non-trivial kernel, because otherwise, $X\in \cD(A)$ which
  is an open set.

  To prove the converse, suppose $A-p(X)$ is positive semidefinite and has
  a non-trivial kernel. Clearly, $X \not \in \cD(A)$.
  For  positive integers $n$, let $A_n=\left(\frac{n+1}{n}\right) A$.
  Then $A_n-p(X)$ is positive definite. Hence
  $X\in \cD(A_n)$ and by convexity of $\cD(A_n)$, for a fixed $0<s<1$,
  $A_n-p(sX)\succ 0$. Letting $n$ tend to infinity, it follows that
  $A-p(sX)\succeq 0$.

  Consider the function $\psi: \R \rightarrow \R$ defined by
  $\psi(t) = \mbox{det}(A - p(tX))$.
  From what has already been proved, $\psi(t)\ge 0$
   for $0\le t<1$. Since $\psi(t)$ is a polynomial
  in the variable $t$, either it vanishes everywhere on $\R$ or
  only on a finite subset of $\R$. If $\psi(t)$ vanishes
  everywhere, then $\psi(0) = \det(A) = 0$ which contradicts the positive
  definiteness of $A$.  Thus $\psi(t)>0$ except
  for finitely many point in $(0,1)$ and
  thus there is a sequence $(s_n)$ from  $(0,1)$
  such that each $s_n X\in \cD(A)$ and
   $s_nX \rightarrow X$. Hence
    $X$ is in the boundary of $\cD(A)$.
\end{proof}

 Suppose that $X \in \smatng$ is in the boundary of $\cD(A)$ and $v \neq 0$
 is a vector in $\R^n$ such that $Av = p(X)v$.

\begin{proposition}
 \label{prop:maintool}
  With $X$ and $v$ as above, if the dimension of the kernel of
 $A-p(X)$ is one, then there exists a subspace $\cH$ of
 $\cT(X,v)$ of codimension one (in $\cT(X,v)$)
  such that, for $H\in \cH$,
 \begin{equation}
  \label{eq:poshess}
   \langle p^{\prime\prime}(X)[H]v,v\rangle \ge 0.
 \end{equation}
\end{proposition}

\begin{remark}\rm
 % \label{rem:moreistrue}
   Since $\cT(X,v)$ has codimension at most $n$ in $\smatng$,
  it turns out that $\cH$ will have codimension at most $n+1$
  in $\smatng$. In fact, a slight modification of the proof
  below shows that there is a subspace $\cK$ of $\smatng$
  of codimension at most $n$ for which equation \eqref{eq:poshess} holds.
  The key point is, with $\Lambda$ as in
  the proof of the proposition,  if $\Lambda(H)=0$,
  then $\langle p^\prime(X)[H]v,v\rangle =0$.

  Unlike a related argument in \cite{DHM1},
  the proof here does not rely on choosing a curve lying
  in the boundary of a convex set, thus eliminating the need
  for a corresponding smoothness hypothesis.
\end{remark}

\begin{proof}
 Since $X$ is in the boundary of the convex set
 $\cD(A),$  there is a linear functional
 $\Lambda:\smatng \to \mathbb R$ such that $\Lambda(Z)<1$
 for $Z\in \cD(A)$ and $\Lambda(X)=1$.
  The subspace
\[
  \cH=\{H \in \cT(X,v) : \Lambda(H)=0 \}
\]
  has codimension one in $\cT(X,v)$.

  Fix $H\in\cH$ and define $F:\mathbb R\to \smatn$ by
  $F(t)=A-p(X+tH)$.  Thus, $F$ is a matrix-valued polynomial in
  the real variable $t$.
  Let $[v]$ denote the one dimensional subspace of $\mathbb R^n$
   spanned by the vector $v$. Write $F(t)$, with respect to the orthogonal
    decomposition of $\mathbb R^n$ as $[v]^\perp \oplus [v]$, as
\[
  F(t)=\begin{pmatrix} Q(t) & g(t) \\ g(t)^* & f(t) \end{pmatrix},
\]
  where $Q$ is a square matrix-valued polynomial, $g$ is a vector, and $f$ is a
  scalar-valued polynomial.  The assumption
  $(A-p(X))v=0$ implies that $f$ and $g$ vanish at $0$.
  The further assumption that $H\in\cT(X,v)$ implies that
  $f$ and $g$ actually vanish to second order at $0$.
  In particular, there are polynomials $\beta$ and $\gamma$ such that
  $g(t)=t^2\beta(t)$ and $f(t)=t^2\gamma(t)$.

  Observe that
\[
  \gamma(0) = - \langle p^{\prime\prime}(X)[H]v,v\rangle.
\]
  Thus, to complete the proof of the theorem it suffices to
  use the choice of $\Lambda$ (and thus the convexity
  hypothesis on $\cD(A)$) and the
  assumption on the dimension of the kernel
  of $A-p(X)$  to show that $\gamma(0)\le 0$.
  Indeed, since the kernel of $A-p(X)$ has dimension one, it follows that
  $Q(0) \succ 0$. Therefore, there exists an
  $\epsilon >0$ such that if $|t| < \epsilon$, then
  $Q(t) \succ 0$. On the other hand,  $\Lambda(X + tH) =
  \Lambda(X) = 1$ for all $t$.  Thus $X + tH \not \in \cD(A)$
  which means
  $F(t) = A - p(X + tH) \not \succ 0$.
  Hence,  the Schur complement of $F$ is nonpositive; i.e.,
\[
  t^2[\gamma(t) - t^2 \beta^*(t)Q^{-1}(t)\beta(t)] \le 0.
\]
 It follows that, for $|t| <\epsilon$,
\[
 \gamma(t) \le t^2 \beta^*(t)Q^{-1}(t)\beta(t)
\]
 and hence $\gamma(0)\le 0$.
\end{proof}

 We end this section with  the following simple observation.

 \begin{lemma}
 \label{lem:generic}
 Suppose $X \in \smatng$ and $0\ne v\in\mathbb R^n$.
  If there is a $\lambda>0$ such that $p(X)v=\lambda v$,
 then there exists a positive definite
 $A \in \smatn$ such that $X$ is in the boundary of $\cD(A)$ and
 $v$ spans the kernel of $A-p(X)$.  Hence, for the triple
 $(A,X,v)$ the conclusion of Proposition
 \ref{prop:maintool} holds.

 Further, if,  for a given positive definite $A \in \smatn$, $X$ is in
  the boundary of $\cD(A)$, and $v$ is a nonzero
  vector such that $(A-p(X))v=0$,
   then for each $\epsilon >0$ there is a $A_\epsilon >0$
  such that $\|A-A_\epsilon\|<\epsilon$, $X$ is in the boundary
  of $\cD(A_\epsilon)$ and the kernel of $(A_\epsilon -p(X))$
  is spanned by $v$.
\end{lemma}

\begin{proof}
  With respect to the decomposition of $\mathbb R^n$
  as $[v]\oplus [v]^\perp$,
\[
  p(X) = \begin{pmatrix} \lambda & 0 \\ 0 & T \end{pmatrix}
\]
 for some symmetric matrix $T$.  Choose $\mu>0$ so that $\mu-T\succ 0$
 and let
\[
  A=\begin{pmatrix} \lambda & 0 \\ 0 & \mu \end{pmatrix}.
\]
  In particular,
\[
  A-p(X) = \begin{pmatrix} 0 & 0 \\ 0 & \mu - T\end{pmatrix}
\]
 is positive semidefinite with one dimensional kernel spanned
 by $v$. From Lemma \ref{lem:closint}, $X$ is in the boundary
 of $\cD(A)$.

 As for the further statement, diagonalize
 with respect to same orthogonal decomposition of $\mathbb R^n$
 as above
\[
  A-p(X) = \begin{pmatrix} 0 & 0 \\ 0 & T \end{pmatrix},
\]
 for some  positive semidefinite $T$. Let $P$ denote
 the projection onto $[v]^\perp$ and let $A_\epsilon = A+\epsilon P$,
 then
\[
 A_\epsilon - p(X) = \begin{pmatrix} 0 & 0 \\ 0 & \epsilon + T \end{pmatrix}
\]
 and the result follows.
\end{proof}
%----------------------------------------------------------------------------

\section{Direct Sums and Linear Independence}
\label{sec:DSLI}
  As in Section \ref{sec:SoS},
  let  $\cK(n)$ denote the set of those $X\in\smatng$
  such that  $p(X)$ has a positive eigenvalue.
  %In view of Lemma \ref{lem:reduce}, we may - and do - assume
  From here on, it is assumed that $-p$ is not a sum of squares.
  Equivalently, $\cK(m)$ is not empty for some $m$.

  Let $\hcK(n)$ denote the set of pairs $(X,v)$
  such that $X\in\cK(n)$ and $v$ is an eigenvector
  of $p(X)$ corresponding to a positive eigenvalue.
  By Lemma \ref{lem:generic},  if $(X,v)\in\hcK(n)$,
  then there exists a positive definite $A\in\smatn$
  such that $X$ is in the boundary
  of $\cD(A)$ and the kernel of $A-p(X)$ is spanned by $v$.
  Let $\langle x\rangle_k$ denote the set of words of length at most $k$.

\begin{lemma}
 \label{lem:preind}
  Fix a positive integer $k$.
  Given $X\in\smatng$ and $v\in\mathbb R^n$, there is
  a nonzero $q\in\Rx$ of degree at most $k$ such that
    $q(X)v=0$ if and only if the set
  $\{w(X)v : w\in \langle x \rangle_k\}$ is linearly dependent.

  If $q\in\Rx$ and $q(X)v=0$ for all $n$ and $(X,v)\in\hcK(n)$, then
  $q=0$.
\end{lemma}

\begin{proof}
  The first statement is evident.  As for the second, the hypotheses
  imply that $\det(q(X))=0$ for each $n$ and $X\in\cK(n)$. Hence
  by Lemma \ref{lem:reduce}, $q=0$.
\end{proof}

\begin{lemma}
 \label{lem:ind}
   Let $d$ denote the degree of $p$.
   Given a positive integer $N$, there
   exists an $n \ge N$ and a pair $(X,v)$ with $X \in \smatng$
   and %, in accordance with Lemma \ref{lem:closint},
    $v \neq 0$ in $\R^n$ such that
 \begin{enumerate}[(i)]
   \item
       there is a subspace $\cH$ of $\cT(X,v)$ of
       codimension at most one
       such that, for all $H\in \cH$,
    \[
    \langle p^{\prime\prime}(X)[H]v,v \rangle \ge 0;
    \]
  \item
     if $q$ is of degree at most $d-1$
     such that $q(X)v=0$, then $q=0$.
 \end{enumerate}
\end{lemma}

\begin{proof}
  Let $\cP$ denote the vector space of polynomials in
  $g$ variables of degree at most $d-1$.
  Given $(Y,w)\in\hcK(n)$, let
\[
  Q(Y,w)=\{q\in\cP: q(Y)w=0\}.
\]
  Thus, $Q(Y,w)$ is a subspace of the finite dimensional vector
  space $\cP$.  Further, by Lemma \ref{lem:preind}
\[
  \cap \{Q(Y,w): (Y,w)\in \hcK(n), \ \ n\in\mathbb N \} = \{0\}.
\]
  Because of finite dimensionality, there
  are positive integers $t$ and $n_1,\dots,n_t$ and
  $(Y^j,w^j) \in \hcK(n_j)$ such that
\[
  \cap_{j=1}^t Q(Y^j,w^j) =\{0\}.
\]
  In particular, if $q\in\Rx$ has degree at most $d-1$ and $q(Y^j)w^j=0$
 for $j=1,\dots,t$,
  then $q=0$.

  Let $Z=\oplus Y^j$ and $z=\oplus w^j$.  Thus $Z$ acts on a space of dimension
  $n^\prime = \sum n_j$. Choose a positive integer $k$ such that
  $n=kn^\prime \ge N$ and let $X=\oplus_1^k Z$ and $v=\oplus_1^k z$.
  From the definition of $\hcK(n)$
  and by Lemma \ref{lem:generic} for each $j$ there is a
  positive definite $A_j\in\mathbb S_{n_j}(\mathbb R)$
  such that $Y^j$ is in the boundary of $\cD(A_j)$ and
   $(A_j-p(Y^j))w^j=0.$ Let $B=\oplus A_j$ and $A^\prime=\oplus_1^k B$.
  Then $(A^\prime-p(X))v=0$ and $A^\prime-p(X)\succeq 0$.
  Moreover, if $q$ has degree at most
  $d-1$ and $q(X)v=0$, then $q=0$.

  Finally, choose a positive definite
  $A\in \smatn$ by the second part of Lemma \ref{lem:generic} such
  that $X$ is in the boundary of $\cD(A)$
  and  the kernel of $(A-p(X))$ is spanned by
  $v$ . In particular, $X$ is in the boundary
  of $\cD(A)$.  The triple $(A,X,v)$ satisfies the hypotheses of
  Proposition
  \ref{prop:maintool}. Hence there is a subspace $\cH$
  of $\cT(X,v)$ of codimension at most one  such that
 \[
   \langle p^{\prime\prime}(X)[H]v,v\rangle \ge 0
 \]
  for all $H\in \cH$.
\end{proof}

  The symmetric polynomial
\[
  r(x)[h] = p^{\prime\prime}(x)[h]
\]
  in the $2g$ variables $(x_1,\dots,x_g,h_1,\dots,h_g)$
  is homogeneous of degree two in $h$.  It admits a representation
  of the form
\[
  r(x)[h] = \begin{bmatrix} V_0(x)[h]^T \cdots V_{d-2}(x)[h]^T \end{bmatrix} Z(x)
       \begin{bmatrix} V_0(x)[h] \\ \vdots \\ V_{d-2}(x)[h]\end{bmatrix}
\]
  where $Z(x)$ is a (uniquely determined square symmetric) matrix of
  free polynomials and $V_j(x)[h]$ is the vector with entries
  $h_\ell w$ over free words $w$ of the the variables $x_1,\dots,x_g$
  of length $j$ and $1\le \ell\le g$.
  (For details see \cite{DHM2}.)  The matrix $\mathcal Z=Z(0)$
  is the \df{middle matrix}.

  \begin{lemma}
   \label{lem:cZ}
   If $\mathcal Z$ is positive semidefinite, then $p$ has degree at most two and
   moreover, $p$ has the form in Equation \eqref{eq:wow}.
  \end{lemma}

   A proof can be found in \cite{HM2}. The idea is that the middle
  matrix $\mathcal Z$ has an {\it antidiagonal} structure which implies,
  if it is positive semidefinite, then its only nonzero entries 
  correspond to $V_0[h]$, which is linear in $h$ and independent 
  of $x$. Thus,
\[
  r(x)[h] = r[h] = V_0[h]^T \cZ V_0[h]
\]
  and it can be shown that $\cZ$ must be positive semidefinite.
  Writing $\cZ$ as a sum of squares and using
\[
   p(x) = \ell(x) +\frac12 r(x)[x]
\]
  expresses $p$ in the form of Equation \eqref{eq:wow}.

 % The following Lemma is a part of Lemma 7.2 from \cite{DHM2}.
%
%\begin{lemma}
% \label{lem:dhm2}
%    There is an integer $\nu$ depending only upon the degree $d$ of the polynomial $p$
%    and the number $g$ of variables such  that the following holds. If
% \begin{enumerate}[(i)]
%  \item $n\ge \nu$;
%  \item $X\in \smatng$ and $v\in\mathbb R^n$; and
%  \item for each $H\in \cT(X,v)$ Equation \eqref{eq:poshess} holds; and
%  \item there does not exist a nonzero polynomial $q$ of degree at most $d-1$
%   satisfying $q(X)v=0$,
% \end{enumerate}
%    then $\mathcal Z$ is positive semidefinite.
%\end{lemma}

 The following Lemma is a consequence of Lemma 7.2 from \cite{DHM2}.

\begin{lemma}
 \label{lem:dhm2}
    There is an integer $\nu$ depending only upon the degree $d$
     of the polynomial $p$ and the number $g$ of variables such
     that the following holds. If
 \begin{enumerate}[(i)]
  \item $n \ge \nu$ satisfies $\frac{\nu+1}{n} < 1$;
  \item $X \in \smatng$ and $v\in\mathbb R^n$; and
  \item there exists a subspace $\cH$ of $\cT(X,v)$ of codimension
  at most one such that
  for each $H\in \cH$ Equation \eqref{eq:poshess} holds; and
  \item there does not exist a nonzero polynomial $q$ of
  degree at most $d-1$ satisfying $q(X)v=0$,
 \end{enumerate}
    then $\mathcal Z$ is positive semidefinite.
\end{lemma}

  To prove Theorem \ref{thm:main} simply observe that the
  existence of an $n$ that satisfies the conditions (i) - (iv)
   of Lemma \ref{lem:dhm2} is guaranteed by Lemma \ref{lem:ind}. The
  conclusion $\cZ$ is positive semidefinite
  combined with Lemma \ref{lem:cZ} now completes
  the proof.

  %To prove Theorem \ref{thm:main} simply observe that
%  the conclusion of Lemma \ref{lem:ind} matches the
%  hypotheses of Lemma \ref{lem:dhm2} with the conclusion
%  that $\cZ$ is positive semidefinite. An application
%  of Lemma \ref{lem:cZ} thus completes the proof.

%  To complete the proof of  Theorem \ref{thm:main}, we appeal
%  to Lemma 7.2 in \cite{DHM2}.
%  According to this lemma,
%  if there exists an integer $\nu$
%  such that for arbitrarily large $n$ there exists
%  a pair $X\in\smatng$ and $v\in\mathbb R^n$
%  and a subspace $\cH$ of $\cT(X,v)$ of codimension
%  at most $\nu$ such that for $H\in\cH$, inequality
%  \eqref{eq:poshess} holds and if
%   $q$ is a polynomial of degree at most $d-1$
%   and $q(X)v=0,$ then $q=0$, then the degree of $p$ is
%  at most two and moreover there is a positive definite
%  $g\times g$ matrix $P$ such that
%\[
%  p^{\prime\prime}(x)[h] = h^* P h = \sum h_j h_k P_{j,k}
%= p^{\prime\prime}(0)[h].
%\]
%  It then follows that,
% % in the notation of equation \eqref{eq:def-dervs},
%  $p= \ell(x) + \frac{1}{2} p^{\prime\prime}(0)[x]$
%  for some linear polynomial $\ell$ and
%  and diagonalizing $P$ puts $p$ in the form
%  of \eqref{eq:wow}.  Thus, combining Lemma 7.2 in \cite{DHM2}
%  and  Lemma \ref{lem:ind} completes the proof of
%  Theorem \ref{thm:main}.

 It remains to prove Theorem \ref{thm:jc}.
 The equivalence of conditions (i) and (ii)
 is the main result of \cite{HM2}.  That (ii)
  implies (iii) is easily checked.
  If (iii) holds, then, by definition,
  $p(x)-p(0)$ is quasi-convex.  Moreover,
   by the second part of Theorem \ref{thm:main},
  condition (iii) implies $p(0)-p(x)$
  is not a nonzero sum of squares. Hence
  (iii) implies (iv).  If (iv) holds,
  then Theorem \ref{thm:main} implies
  $p(x)-p(0)$ is a linear term plus
  a (hermitian) sum of squares of linear polynomials
  and thus (i) holds.


\begin{thebibliography}{99}

\bibitem{BGM}  Ball, Joseph A.; Groenewald, Gilbert; Malakorn, Tanit,
 {\it Structured noncommutative multidimensional linear systems.}
   SIAM J. Control Optim. 44 (2005), no. 4, 1474--1528.

\bibitem{C} Cimpric, Jakob,
 {\it Noncommutative Positivstellens\"atze for pairs representation-vector,}
   Positivity 15 (2011) 481-495.

\bibitem{DHM1} Dym, Harry; Helton, William; McCullough, Scott,
Irreducible noncommutative defining polynomials for convex sets have degree four or less.
Indiana Univ. Math. J. 56 (2007), no. 3, 1189--1231.

\bibitem{DHM2} Dym, Harry;  Helton, J. William;  McCullough, Scott,
 {\it Non-commutative Varieties with
Curvature having Bounded Signature.}  Illinois Journal of
 Mathematics (to appear),  arXiv:1202.0056.

\bibitem{DHM3} Dym, Harry; Helton, J. William; McCullough, Scott,
{\it The Hessian of a noncommutative polynomial has numerous negative eigenvalues.}
J. Anal. Math. 102 (2007), 29--76.

\bibitem{EW} Effros, Edward G.; Winkler, Soren,
 {\it Matrix convexity: operator analogues of the bipolar and Hahn-Banach theorems.}
J. Funct. Anal. 144 (1997), no. 1, 117--152.

\bibitem Helton, J. William,
{\it ``Positive” noncommutative polynomials are sums of squares,}
Ann. Math. 156 (2002) 675-694.

\bibitem{HHLM} Hay, Damon M.; Helton, J. William;
Lim, Adrian; McCullough, Scott,
Non-commutative partial matrix convexity.
Indiana Univ. Math. J. 57 (2008), no. 6, 2815--2842.

%\bibitem{HKM} Helton, J. William; Klep, Igor;  McCullough Scott,
%Free Convex Algebraic Geometry - Tutorial - Arxiv??

\bibitem{HM} Helton, J. William; McCullough, Scott,
 {\it Every free basic convex semi-algebraic
 set has an LMI representation,} to appear in The Annals of Math.

\bibitem{HM2}  Helton, J. William; McCullough, Scott,
 {\it Convex Noncommutative Polynomials Have Degree Two or Less,}
  SIAM J. Matrix Anal. Appl. 25 (2004) 1124-1139.

\bibitem{HMP}  Helton, J. William; McCullough, Scott; Putinar, Mihai,
 {\it Strong majorization in a free $*$-algebra.}
   Math. Z. 255 (2007), no. 3, 579--596.

\bibitem{HMPV} Helton, J. William; McCullough, Scott; Putinar, Mihai;
  Vinnikov, Victor,
{\it Convex matrix inequalities versus linear matrix inequalities.}
 IEEE Trans. Automat. Control 54 (2009), no. 5, 952--964.

\bibitem{DKV-VV}  Kaliuzhnyi-Verbovetskyi, Dmitry; Vinnikov, Victor,
  work in progress.

\bibitem{OHMP}  de Oliveira, Mauricio C.; Helton, J. William; McCullough, Scott A.); Putinar, Mihai,
{\it Engineering systems and free semi-algebraic geometry.}
   Emerging applications
 of algebraic geometry, 17--61, IMA Vol. Math. Appl., 149, Springer,
  New York, 2009.

\bibitem{KS1}  Klep, Igor; Schweighofer, Markus,
 {\it A nichtnegativstellensatz for polynomials in noncommuting variables.}
  Israel J. Math. 161 (2007), 17--27.

\bibitem{KS2} Klep, Igor; Schweighofer, Markus,
  {\it  Sums of Hermitian squares and the BMV conjecture.}
   J. Stat. Phys. 133 (2008), no. 4, 739--760.

\bibitem{M} McCullough, Scott,
 {\it Factorization of operator-valued polynomials in several non-commuting variables.} Linear Algebra Appl. 326 (2001), no. 1-3, 193--203.

\bibitem{MS}  Muhly, Paul S.; Solel, Baruch,
{\it Schur class operator functions and automorphisms of Hardy algebras.}
  Doc. Math. 13 (2008), 365--411.

\bibitem{Pop}  Popescu, Gelu,
 {\it  Free holomorphic functions on the unit ball of $B(H)^n$ II.}
  J. Funct. Anal. 258 (2010), no. 5, 1513--1578.

\bibitem{Shu}  Sch\"utzenberger, M. P.
On the definition of a family of automata.
Information and Control 4 1961 245--270. %  schutzenberger

\bibitem{S}  Schm\"udgen, Konrad,
{\it Unbounded operator algebras and representation theory,}
  Operator Theory: Advances and Applications, 37. Basel etc.: Birkh\"auser Verlag.
1989.

\bibitem{Sima} Schm\"udgen, Konrad,
 {\it Noncommutative Real Algebraic Geometry Some Basic Con-
cepts and First Ideas,} Emerging Applications of Algebraic Geometry, The IMA
Volumes in Mathematics and its Applications Volume 149, 2009, 325-350.

\bibitem{Voic} Voiculescu, Dan-Virgil,
 {\it Free analysis questions II: the Grassmannian completion and the series expansions at the origin.} J. Reine Angew.
   Math. 645 (2010), 155--236.
\end{thebibliography}
 \end{document}